\documentclass[pdflatex,sn-mathphys-num]{sn-jnl}% Math and Physical Sciences Numbered Reference Style

 \usepackage{amssymb,amsthm,amsmath}
\usepackage{amscd} % allows index generation
\usepackage{relsize}

 \usepackage{graphicx,color} % standard LaTeX graphics tool
\usepackage{url} % <-- Para p\'aginas web o similar: \url{...}

\theoremstyle{thmstyleone}%
\newtheorem{theorem}{Theorem}[section]%  meant for continuous numbers
\newtheorem{proposition}[theorem]{Proposition}% 
\newtheorem{corollary}[theorem]{Corollary}

\theoremstyle{thmstyletwo}%
\newtheorem{remark}[theorem]{Remark}%

\theoremstyle{thmstylethree}%
\newtheorem{definition}[theorem]{Definition}%

\def\r{\mathbb R}

\def\s{\mathbb S}

\begin{document}

\title[A connection between minimal surfaces and a problem of Euler]{A connection between minimal surfaces and the two-dimensional analogues of a problem of Euler}


\author{\fnm{Rafael} \sur{L\'opez}}\email{rcamino@ugr.es}

\affil{\orgdiv{Department of Geometry and Topology}, \orgname{University of Granada}, \orgaddress{ \city{Granada}, \postcode{18071},  \country{Spain}}}


\abstract{If $\alpha\in\r$, an $\alpha$-stationary surface in Euclidean space is a surface $\Sigma$ whose mean curvature $H$ satisfies $H(p)=\alpha |p|^{-2} \langle\nu,p\rangle$, $p\in\Sigma$. These surfaces generalize in dimension two a classical family of curves studied by Euler which are critical points of the moment of inertia of planar curves. In this paper we establish, via inversions, a one-to-one correspondence between $\alpha$-stationary surfaces and $-(\alpha+4)$-stationary surfaces. In particular, there is a correspondence between $-4$-stationary surfaces and minimal surfaces. Using this duality we give some results of uniqueness of $-4$-stationary surfaces and we solve the B\"{o}rling problem. }




\maketitle

%%%%%%%%%%%%%%%%%%%%
\section{Introduction and statement of the results}
%%%%%%%%%%%%%%%%%%%%%%%%%%%

In its magna opus,  Euler studied the minimization of the moment of inertia of planar curves \cite{eu}. Taking the origin as reference for the computation of the moment of inertia, consider all planar curves joining two given points. In the language of theory of variations, the purpose it to minimize the energy functional 
\begin{equation}\label{y1}
[y]\mapsto \int_a^b(x^2+y^2)\sqrt{1+y'^2}
\end{equation}
 for all functions $y=y(x)$, $x\in [a,b]$, where $y(a)=y_0$ and $y(b)=y_1$ and $y_0,y_1$ are given. Recently, Dierkes and Huisken have generalized this problem in arbitrary dimensions \cite{dh}. For our purposes, we only treat the two dimensional case. Let $\Sigma$ be a connected, oriented surface and consider a smooth immersion of $\Sigma$ in the Euclidean $3$-space $(\r^3,\langle,\rangle)$. Introducing a parameter $\alpha\in\r$, we define the analog of the energy \eqref{y1} for surfaces and define 
$$E_\alpha[\Sigma]=\int_\Sigma |p|^\alpha\, d\Sigma,$$
where $d\Sigma$ denotes the area element induced on $\Sigma$ and we   identify $p\in\Sigma$ with its image by the immersion. We find the critical points of this functional by using standard arguments of calculus of variations for compactly supported variations of $\Sigma$. In particular, and for the differentiability of $E_\alpha$, we need that $\Sigma$ does not contain the origin $0$ of $\r^3$. Then critical points of $E_\alpha$ are characterized by the equation 
\begin{equation}\label{eq1}
H(p)=\alpha\frac{\langle \nu(p),p\rangle}{|p|^2},\quad p\in\Sigma.
\end{equation}
Here $H$ and $\nu$ are the mean curvature and the unit normal vector of $\Sigma$, respectively. The convention for the mean curvature is that $H$ is the sum of the principal curvatures, where a sphere of radius $r>0$ has $H=2/r$ with respect to the inward orientation.

\begin{definition} A surface $\Sigma$ of $\r^3-\{0\}$ is said to be an $\alpha$-stationary surface if $\Sigma$ satisfies Eq. \eqref{eq1}.
\end{definition}
In particular, $0$-stationary surfaces are minimal surfaces. In \cite{dh}, and for arbitrary dimensions, Dierkes and Huisken studied stability of spheres and minimal cones as well as minimizers of $E_\alpha$. Dierkes together the author have recently obtained new results on stationary surfaces. For example, existence of the Plateau problem \cite{dl1}, characterizations of the axisymmetric surfaces \cite{dl2}, ruled stationary surfaces \cite{lo1} or investigating those stationary surfaces with constant Gauss curvature \cite{lo2}.

Stationary surfaces also are minimal surfaces in a manifold endowed with a density. Indeed, given a positive density $\phi\in C^\infty(\r^3)$, let $dV_\phi=\phi\, dV_0$ and $dA_\phi=\phi\, dA_0$ be the weighted volume and area, where $dV_0$ and $dA_0$ are the Euclidean volume and area of $\r^3$. A surface $\Sigma$ is a critical point of $A_\phi$ if and only if the weighted mean curvature $H_\phi$ vanishes on $\Sigma$, where $H_\phi$ is 
\begin{equation*}
H_\phi= H- \langle\nu,\overline{\nabla}\phi \rangle,
\end{equation*}
and $\overline{\nabla}$ is the gradient in $\r^3$ \cite{ba}. If we choose the function $\phi(p)=|p|^\alpha$ defined in $\r^3-\{0\}$, then $H_\phi=0$ coincides with Eq. \eqref{eq1}. As a consequence, the class of $\alpha$-stationary surfaces satisfies a tangency principle such as it occurs with minimal surfaces: see Prop. \ref{pr-pt} in Sect. \ref{s2}. 

Among all values of $\alpha$, only at two ones there exist round spheres that are $\alpha$-stationary surfaces. To be precise, spheres centered at $0$ are $-2$-stationary surfaces and spheres containing $0$ are $-4$-stationary surfaces. This particular property of these spheres makes that the values $-2$ and $-4$ have a special role in the theory of $\alpha$-stationary surfaces.

Consider the inversion map of $\r^3$, 
\begin{equation}\label{in}
\Phi\colon\r^3-\{0\}\to\r^3-\{0\},\quad \Phi(p)=\frac{p}{|p|^2}.
\end{equation}
We prove the following result that establishes a correspondence, or duality, between stationary surfaces for different values of $\alpha$.

\begin{theorem} \label{tp}
The map $\Phi$ carries $\alpha$-stationary surfaces in $-(\alpha+4)$-stationary surfaces. 
\end{theorem}
 For the value $\alpha=0$, we deduce: 
 
\begin{corollary}\label{c1} In Euclidean space $\r^3$, there is a one-to-one correspondence between minimal surfaces and $-4$-stationary surfaces.
\end{corollary}

 Theorem \ref{tp} will be proved in Sect. \ref{s3}. As a consequence of this duality, it is possible to address problems of $-4$-stationary surfaces via its formulation in the theory of minimal surfaces. In this paper we exploit this correspondence to obtain the following results. 

\begin{theorem}\label{t2}
Planar discs of vector planes and spherical caps of spheres passing through $0$ are the only compact $-4$-stationary surfaces with circular boundary
\end{theorem}

\begin{theorem}\label{t3} Spheres passing through $0$ are the only $-4$-stationary surfaces properly immersed in $\r^3$ and contained in a vector halfspace.
\end{theorem}

Theorem \ref{t2}  will proved in Sect. \ref{s4} while Thm. \ref{t3}  in Sect. \ref{s5}.  Finally in Sect. \ref{s6} we formulate the Bj\"{o}rling problem giving a solution. As a consequence, we will show examples  $-4$-stationary surfaces with the topology of  a M\"{o}bius strip.

%%%%%%%%%%%%%%%%%%%%%%%%%%%%%%%%%%%
\section{Preliminaries}\label{s2}
%%%%%%%%%%%%%%%%%%%%%%%%%%%%%

In this section we show some examples of stationary surfaces, give the tangency principle and some applications. 

In the definition of the energy functional $E_\alpha$, the value $|p|$ is the distance of $p\in\Sigma$ to the origin $0\in\r^3$. This implies that the corresponding Euler-Lagrange equation \eqref{eq1} is not preserved, in general, by rigid motions of $\r^3$, as for example translations. So, if $\Sigma$ is an $\alpha$-stationary surface and $T\colon\r^3\to\r^3$ is a translation, then $T(\Sigma)$ is not an $\alpha$-stationary surface. However, vector isometries as well as dilations from $0\in\r^3$, preserve solutions of \eqref{eq1} for the same value of $\alpha$. 

We give examples of $\alpha$-stationary surfaces. For this, we focus in surfaces of $\r^3$ with constant mean curvature, to be precise, planes and spheres, and we ask which ones are stationary surfaces. We point out that cylinders are not stationary surfaces. The following result is straightforward.

\begin{proposition}\label{pr-ex}
\begin{enumerate}
\item A plane is an $\alpha$-stationary surface if and only if it is a vector plane. This occurs for all $\alpha\in\r$.
\item The only stationary spheres are spheres centered at $0$ ($\alpha=-2$) and spheres containing $0$ ($\alpha=-4$). 
\end{enumerate}
\end{proposition}

As we said in the Introduction, the values $-2$ and $-4$ will play a special role in the range $\alpha$ of $\alpha$-stationary surfaces. An example is in the study of closed (compact without boundary) $\alpha$-stationary surfaces. In \cite[Thm. 1.6]{dh} the authors proved the following result: 
\begin{enumerate}[(i)]
\item If $\alpha>-2$, there are no closed $\alpha$-stationary surfaces.
\item If $\alpha=-2$, the only stable closed stationary surfaces are spheres centered at the origin. 
\item Let $\alpha<-2$. If $\Sigma$ is a (non-extendible) $\alpha$-stationary surface, then its closure $\overline{\Sigma}$ must contain the origin $0$ of $\r^3$. In consequence, there are not closed $\alpha$-stationary surfaces
\end{enumerate}
The proof of (i)  involves the use of the Hopf maximum principle for the function $|p|^2$. If $\alpha=-2$, the statement was proved using the expression of the second variation of the energy $E_\alpha$. With a different approach, item (i) can be also proved using the tangency principle of weighted minimal surfaces. Since $\alpha$-stationary surfaces are minimal in the weighted space $(\r^3,|p|^\alpha)$, we have the following consequence of the maximum principle:

\begin{proposition}[Tangency principle]\label{pr-pt}
Let $\Sigma_1$ and $\Sigma_2$ be two (connected) $\alpha$-stationary surfaces and assume that they are tangent at some $p\in \Sigma_1\cap \Sigma_2$, where $p$ is a common interior point or a common boundary point where $\partial\Sigma_1$ and $\partial\Sigma_2$ are tangent at $p$. If $\Sigma_1$ lies at one side of $\Sigma_2$, then $\Sigma_1=\Sigma_2$ in the largest neighborhood of $p$ in $\Sigma_1\cap \Sigma_2$. 
\end{proposition}

We introduce the next notation. Let $\s^2(r)$ denote the sphere with center $0$ and radius $r>0$ and $\s^2_r(r)$ denotes the sphere of radius $r>0$ and centered at $(0,0,r)$. Then $\s^2(r)$ is a $-2$-stationary surface while $\s^2_r(r)$ is a $-4$-stationary surface. 

Let $\Sigma$ be an $\alpha$-stationary closed surface. If we take $r$ sufficiently big, then $\Sigma$ is included in the ball determined by $\s^2(r)$. Decreasing $r$, $r\searrow 0$, we arrive until the first sphere $\s^2(r_1)$ that touches   $\Sigma$, which it is a  common interior point. Then the Tangency principle proves that $\alpha\leq -2$ and in case that $\alpha=-2$, then $\Sigma=\s^2(r_1)$ (without any assumption on stability of $\Sigma$). This proves (i) and (ii) of the Dierkes-Huisken's theorem. However, Theorem \ref{tp} allows to prove the non-existence of closed $\alpha$-stationary surfaces for $\alpha<-2$ if one knows the result when $\alpha>-2$.

\begin{corollary} There are no closed $\alpha$-stationary surfaces if $\alpha<-2$.
\end{corollary}

\begin{proof} Inversions preserve closedness of surfaces. On the other hand, if $\Sigma$ is an $\alpha$-stationary surface for $\alpha<-2$, then $\Phi(\Sigma)$ is a $-(4+\alpha)$-stationary surface where now $-(\alpha+4)>-2$. Since $\Phi(\Sigma)$ cannot be closed, neither $\Sigma$.
\end{proof}

%%%%%%%%%%%%%%%%%%%%%%%%%%%%%%%%%%%
\section{Proof of Theorem \ref{tp}}\label{s3}
%%%%%%%%%%%%%%%%%%%%%%%%%%%%%

The proof of Thm. \ref{tp} is  based on the computation of the mean curvature of an inverse surface. Consider the inversion map $\Phi$ defined in \eqref{in}. If $\Sigma$ is a surface of $\r^3$, $0\not\in\Sigma$, then $\widetilde{\Sigma}:=\Phi(\Sigma)$ is a surface and there is a relation between the principal curvatures of $\Sigma$ and $\widetilde{\Sigma}$. Let $\kappa_1$ and $\kappa_2$ be the principal curvatures of $\Sigma$ with respect to the unit normal $\nu$. Let $h\colon\Sigma\to\r$ be the support function given by $h(p)=\langle \nu(p),p\rangle$. We denote by tildes $\sim$ the corresponding objects in $\widetilde{\Sigma}$, that is, $\widetilde{\kappa_i}$, $\widetilde{\nu}$, and so on. Then it is known that 
\begin{equation}\label{nn}
\begin{split}
\widetilde{\kappa}_i\circ\Phi(p)&=|p|^2\kappa_i(p)+2h(p),\\
 \widetilde{\nu}(p)&=\nu(p)-2h(p)\frac{p}{|p|^2}.
 \end{split}
 \end{equation}
 See for example \cite[Ch. 3. ex. 15]{mr} or \cite{wh}. In consequence, 
 $$\widetilde{H}\circ\Phi(p)=|p|^2H(p) + 4h(p).$$
We express this identity in terms of the surface $\widetilde{\Sigma}$. We have 
$$h=-\frac{\widetilde{h}}{|\widetilde{p}|^2},\quad |p|^2=\frac{1}{|\widetilde{p}|^2}.$$
Therefore
\begin{equation}\label{du}
\widetilde{H}\circ\Phi=\frac{H-4\widetilde{h}}{|\widetilde{p}|^2}.
\end{equation}

If $\Sigma$ is an $\alpha$-stationary surface, then its mean curvature $H$ satisfies \eqref{eq1}. Hence, and thanks to \eqref{du}, the mean curvature $\widetilde{H}$ of its inverse surface $\widetilde{\Sigma}$ fullfils
$$\widetilde{H}\circ\Phi=\frac{\alpha h/|p|^2-4\widetilde{h}}{|\widetilde{p}|^2}=\frac{-\alpha\widetilde{h}-4\widetilde{h}}{|\widetilde{p}|^2}=-(\alpha+4)\frac{\widetilde{h}}{|\widetilde{p}|^2}.$$
This proves that $\widetilde{\Sigma}$ is a $-(\alpha+4)$-stationary surface. The converse is analogous and this finishes the proof of Thm. \ref{tp}

As one can see in \cite{dh}, the value $\alpha=-2$ is the frontier between the class of stationary surfaces for $\alpha>-2$ and that of $\alpha<-2$. Theorem \ref{tp} now establishes a connection between both ones, because the function $\alpha\mapsto-(\alpha+4)$ maps the interval $(-\infty,-2)$ into $(-2,\infty)$ leaving invariant the value $\alpha=-2$.

\begin{remark} \label{r1}
Rigid motions of $\r^3$ carry minimal surfaces into minimal surfaces. However, these rigid motions do not carry, in general, stationary surfaces into stationary surface, as for example, a translation $T(p)=p+\vec{v}$, where $\vec{v}\in\r^3$. If $\Sigma$ is a minimal surface, then $T(\Sigma)$ is a minimal surface. Both minimal surfaces give two stationary surfaces via $\Phi$ and Cor. \ref{c1}, which have no a relation between them. In other words, the following diagram   is not commutative. 
 $$
 \begin{CD}
 \{-4-\mbox{stationary surfaces} \}@>\mathlarger{\Phi}>> \{\mbox{minimal surfaces}\}\\
 @V ? V V @VV \mbox{translation} V\\
 \{-4-\mbox{stationary surfaces} \}@< \mathlarger{\Phi}^{-1} << \{\mbox{minimal surfaces}\}
 \end{CD}
 $$
Notice that 
$$\Phi^{-1}\circ T\circ\Phi(p)=\dfrac{p+|p|^2\vec{v}}{1+|p|^2|\vec{v}|^2+2\langle p,v\rangle}.$$
\end{remark}

%%%%%%%%%%%%%%%%%%%%%%%%%
\section{Stationary surfaces with circular boundary}\label{s4}
%%%%%%%%%%%%%%%%%%%%%%%%%%%%%%%%%%%%%%%%%%%%

In this section we give an application of Cor. \ref{c1} in the study of compact stationary surfaces with circular boundary. We focus in the case that the circle spans a spherical cap, so we are considering the cases $\alpha=-2$ and $\alpha=-4$. Let $\Gamma$ be a circle contained in $\s^2(r)$ or in $\s_r^2(r)$. Then $\Gamma$ separates these spheres in two components as follows:
\begin{enumerate}
\item If $\alpha=-2$, then $\s^2(r)-\Gamma$ are two spherical caps. 
\item If $\alpha=-4$, then $\s_r^2(r)-\Gamma$ is formed by a spherical cap and a punctured spherical cap and only the first one is compact. 
\end{enumerate}
On the other hand, any round disc in a vector plane which and not containing $0$ is an $\alpha$-stationary surface spanning a circle.  This surface is compact while if $D$ is a round dis containing $0$ in its interior, then $D_\{0\}$ is a $\alpha$-stationary with circular boundary but it is not compact.

Then the following question arises in a natural way. 

\begin{quote}
Question: Let $\alpha\in\{-2,-4\}$. Are planar discs and spherical caps the only compact $\alpha$-stationary surfaces with circular boundary? 
\end{quote}
This question is similar to that in the theory of constant mean curvature (cmc) surfaces of $\r^3$. Recall that there are examples of non-spherical cmc compact surfaces whose boundary is a circle \cite{ka}. However, the boundary versions of the Alexandrov and Hopf theorems are yet unanswered \cite{lo}.

By the above two examples of spheres, not any circle $\Gamma$ of $\r^3$ spans a spherical cap of $\s^2(r)$ or $\s_r^2(r)$. This is the case, for instance, of the circle $\{z=0, (x-2)^2+y^2=1\}$ which it is contained in the horizontal plane $z=0$. A slightly change of the problem is assuming that $\Gamma$ is contained in some of the above two spheres and asking if spherical caps of such spheres are the only compact $-2$ or $-4$-stationary surfaces. Under such hypothesis, the tangency principle provides an affirmative answer in the case $\alpha=-2$.

\begin{theorem}\label{t22} Let $\Gamma$ be a circle contained in $\s^2(r_0)$. Then the spherical caps of $\s^2(r_0)$ determined by $\Gamma$ are the only compact $-2$-stationary surfaces spanning $\Gamma$.
\end{theorem}
\begin{proof}
 Let $\Sigma$ a compact $-2$-stationary surface spanning a circle $\Gamma$.   If $r\nearrow\infty$, then $\s^2(r)$ does not intersect $\Sigma$ for big values of $r$ because $\Sigma$ is compact. Let $r\searrow r_0$ until the first sphere $\s^2(r_1)$ that touches  $\Sigma$, $r_1\geq r_0$. If the touching point is an interior point or a boundary point where the boundaries of both surfaces are tangent,  then the Tangency principle implies that $\Sigma\subset\s^2(r_1)$, $r_1=r_0$, obtaining the result. Otherwise, $r_1=r_0$ and $\mbox{int}(\Sigma)$ is contained in open ball determined by $\s^2(r_0)$. We see that this situation is not possible. We now repeat the argument with spheres $\s^2(r)$ and $r$ close to $0$ and $\s^2(r)\cap\Sigma=\emptyset$. This is possible because $\Sigma$ is compact and $0\not\in\Sigma$. Increasing $r\nearrow r_0$, then there is a first radius $r_2>0$ such that  $\s^2(r_2)$ touches $\mbox{int}(\Sigma)$, and $r_2<r_0$. The Tangency principle proves that $\Sigma\subset\s^2(r_2)$, hence $r_0=r_2$, which it is a contradiction. 
\end{proof}

Notice that the key in the proof is the fact that the  family of spheres $\{\s^2(r)\colon r>0\}$ centered at $0$ provides a foliation of $\r^3-\{0\}$. In consequence, it is not possible to give a similar proof in the case $\alpha=-4$. For example, and following the same argument of the proof, we do not know if it is possible to find a sphere $\s_r^2(r)$ of sufficiently radius so $\Sigma$ is included in the ball defined by this sphere: the surface $\Sigma$ may have points in both sides of the plane $z=0$. However, in case that the circle $\Gamma$ is contained in a vector plane, we prove the analogous result for $\alpha=-4$. 

\begin{theorem} Let $\Gamma$ be a circle contained in a vector plane and let $D$ be the closed planar disc bounded by $\Gamma$. Assume $0\not\in D$. Then $D$ is the only compact $-4$-stationary surface spanning $\Gamma$.
\end{theorem}

\begin{proof} Let $\Sigma$ be a compact $-4$-stationary surface spanning $\Gamma$. Suppose that $\Sigma$ is not $D$ and we arrive to a contradiction. After a vector isometry, we can assume that $\Gamma$ is contained in the plane $z=0$ and that $\Sigma$ contains points in the halfspace $z>0$. Since $\Sigma$ is compact and $0\not\in D$, let $r>0$ be sufficiently small so $\s^2_r(r)\cap\Sigma=\emptyset$. Using that the family of spheres $\{\s^2_r(r)\colon r>0\}$ gives a foliation of the halfspace $z>0$, and by letting $r\nearrow\infty$, let $r_1>r$ be the first sphere $\s^2_{r_1}(r_1)$ that touches $\Sigma$. This occurs at some interior point of both surfaces and the Tangency principle implies that $\Sigma$ is contained in is possible $\s^2_{r_1}(r_1)$. This is a contradiction because $\s^2_{r_1}(r_1)\cap\{z=0\}=\{0\}$.
\end{proof}

We point that the arguments in the above proof fail if $\alpha=-2$ because the spheres $\s^2(r)$ of the foliation of $\r^3-\{0\}$ are not tangent at $0$, such as it occurs for the spheres $\s^2_r(r)$.

We now come back to the initial question. We prove that the answer is affirmative regardless any assumption of the position of the given circle.

\begin{theorem} 
Planar discs of vector planes and spherical caps of spheres passing through the origin $0\in\r^3$ are the only compact $-4$-stationary surfaces with circular boundary.
\end{theorem}

\begin{proof} Let $\Sigma$ be a compact $-4$-stationary whose boundary is a circle $\Gamma$. We know that $0\not\in\Gamma$. Corollary \ref{c1} establishes that $\widetilde{\Sigma}:=\Phi(\Sigma)$ is a minimal surface whose boundary $\widetilde{\Gamma}$ is a circle because the inversion $\Phi$ carries circles that do not across $0$ in circles, such it happens with the circle $\Gamma$. Thus $\widetilde{\Sigma}$ is a compact minimal surface whose boundary is a circle. However the only compact minimal surfaces spanning planar curves are domains of planes. Then $\widetilde{\Sigma}$ is a planar disc with $0\not\in \widetilde{\Sigma}$. We have two possibilities. Suppose that $\widetilde{\Sigma}$ is contained in a vector plane. This situation occurs when $\Gamma$ is contained in a vector plane. Then $\Sigma$ is a planar disc. In case that $\widetilde{\Sigma}$ is not contained in a vector plane, then its inversion via $\Phi$ is part of a sphere. Since the only $-4$-stationary surfaces that are spheres are that of type $\s^2_r(r)$ (Prop. \ref{pr-ex}), this proves that $\Sigma$ is a spherical cap of a sphere passing through the origin $0\in\r^3$. This completes the proof. 
\end{proof}

It deserves the following observations.   For $\alpha$-stationary surfaces, one may think that a type of Alexandrov reflection technique can be employed to prove that $\Sigma$ is a surface of revolution \cite{al}, even in the case that $\alpha=-4$ or $\alpha=-2$. However, the proof, if exists, is not immediate. In arguments involving the reflection technique by planes, it is necessary to assume that $\Sigma$ is embedded which it is not the case of Thm. \ref{t2}. 
 Even assuming embeddedness, reflections in $\r^3$ about planes cannot be employed because these symmetries are not vector isometries in general and, consequently, they do not preserve solutions of Eq. \eqref{eq1}.

%%%%%%%%%%%%%%%%%%
\section{Further two applications of the duality}\label{s5}
%%%%%%%%%%%%%%%%%%%

Using Cor. \ref{c1}, we classify all ruled $-4$-stationary surfaces. Recall that inversions preserve, as set, the family of curves formed by circles and straight-lines. For the proof we use that the only minimal surfaces of $\r^3$ constructed by a one-parameter family of circles are planes, catenoids and the Riemann minimal examples \cite{en1,en2,ri}.

\begin{theorem} Vector planes are the only ruled $-4$-stationary surfaces.
\end{theorem}

\begin{proof} It is clear that vector planes are ruled surfaces and also are $-4$-stationary surfaces. Suppose now that $\Sigma$ is a non-planar ruled $-4$-stationary surface and we will arrive to a contradiction. A particular case to consider is that $\Sigma$ is a (non-planar) conical surface from $0$. In such a case, the inverse surface $\widetilde{\Sigma}: =\Phi(\Sigma)$ coincides with $\Sigma$. However, conical surfaces are ruled surfaces and the only ruled minimal surfaces are the plane (which was discarded) and the helicoid, which it is not a conical surface.  

Let now $\Sigma$ be a ruled surface   containing straight-lines that do not across $0$. In particular, in an open set of $\Sigma$ the surface is not a cone from $0$. Since the map $\Phi$ carries the straight-lines of $\Sigma$ in circles of $\widetilde{\Sigma}$, we deduce that $\widetilde{\Sigma}$ is a minimal surface and foliated by circles. Moreover, all these circles must contain the origin $0$ because their preimages vi $\Phi$ are straight-lines that do not contain $0$. In such a case it is known that $\widetilde{\Sigma}$ is a plane, a catenoid or one example of the Riemann minimal examples. The plane is not possible because $\Sigma$ is not a plane neither a sphere (a sphere is not a ruled surface). The case that $\widetilde{\Sigma}$ is a catenoid or a Riemann minimal example is neither possible because the circles of these surfaces do not across a common point.  
\end{proof}

 Suppose $\Sigma$ is a minimal surface foliated by circles, that is, $\Sigma$ is a plane, a catenoid or a Riemann minimal example. We know that $\Phi$ carries circles into circles or straight-lines. In case that $\Sigma$ is a plane, then its dual surface $\widetilde{\Sigma}$ a sphere containing $0$ or a vector plane. Both surfaces are, of course, $-4$-stationary surfaces. 

In case that $\Sigma$ is a catenoid or a Riemann minimal example, its inverse surface is foliated by circles because the circles of the two minimal surfaces do not across the origin of $\r^3$ (except perhaps one circle). Suppose that $\Sigma$ is a catenoid. Taking into account Rem. \ref{r1}, each position in the space of $\Sigma$ provides different $-4$-stationary surfaces. In case that the rotation axis of $\Sigma$ acrosses $0$, say the $z$-axis, then $\Phi$ transforms $\Sigma$ into a surface of revolution with the same axis; otherwise, the surface $\widetilde{\Sigma}$ is not rotational, although it is formed by a one-parameter family of circles, with possible straight-lines. 

Since the catenoid goes to infinity, its dual $-4$-stationary surface contains the origin, which it is according to the theory of $-4$-stationary surfaces. In Fig. \ref{fig1} we show two $-4$-stationary surfaces which are dual surfaces of catenoids whose axes are parallel to the $z$-axis. In the left surface, the dual surface is a catenoid whose axis is the $z$-axis and, in consequence, the inversion of this surface is a surface of revolution begin the $z$-axis the rotation axis. In the right surface, the axis of the catenoid is the vertical line through the point $(-\frac12,0,0)$ and thus the dual surface is not rotational. 
 
 In Fig. \ref{fig2} we show the inverse surface of a Riemann minimal example. 
 \begin{figure}[hbtp]
\begin{center}
\includegraphics[width=.5\textwidth]{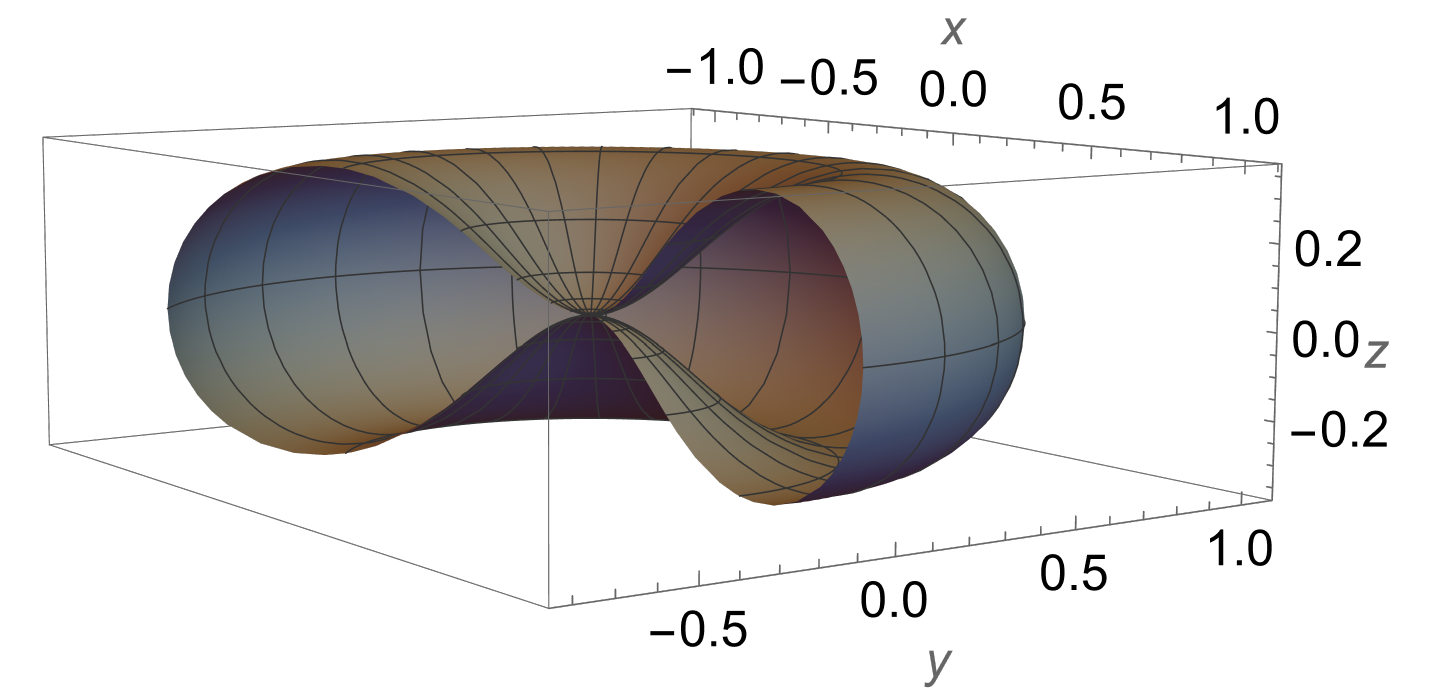}, \includegraphics[width=.4\textwidth]{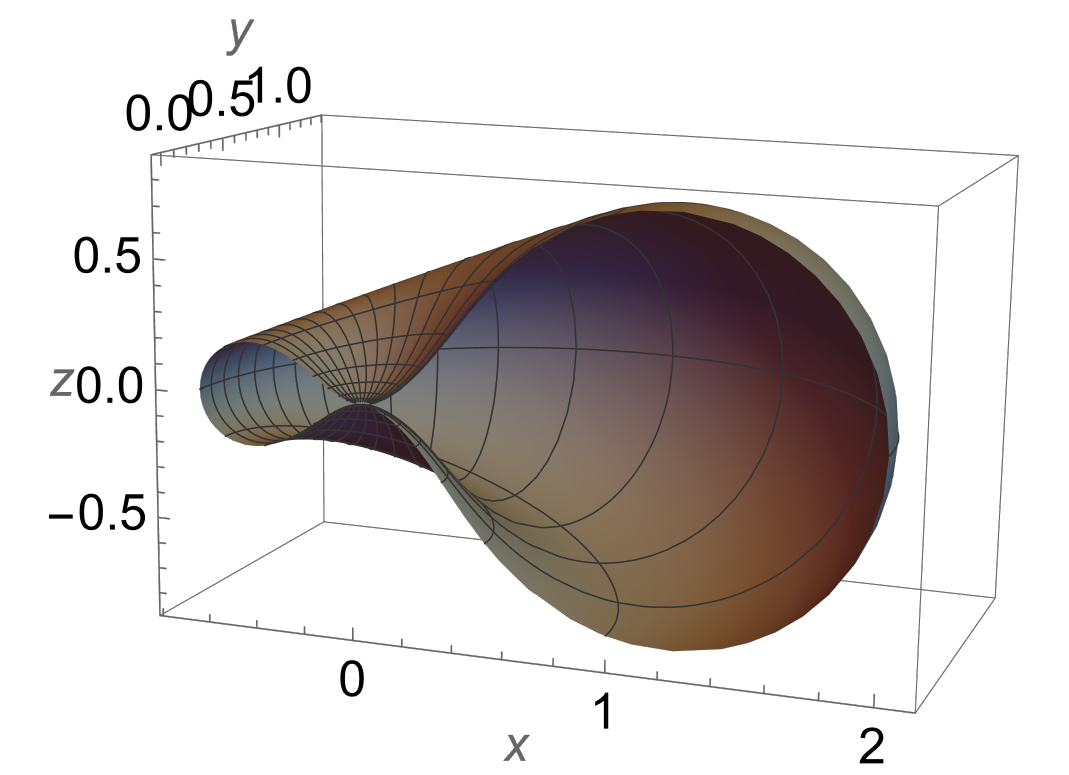}
\end{center}
\caption{Cross-sections of inverse catenoids, where the axis of the catenoid acrosses the origin (left) and does not (right).}\label{fig1}
\end{figure}
 
 \begin{figure}[hbtp]
\begin{center}
\includegraphics[width=.4\textwidth]{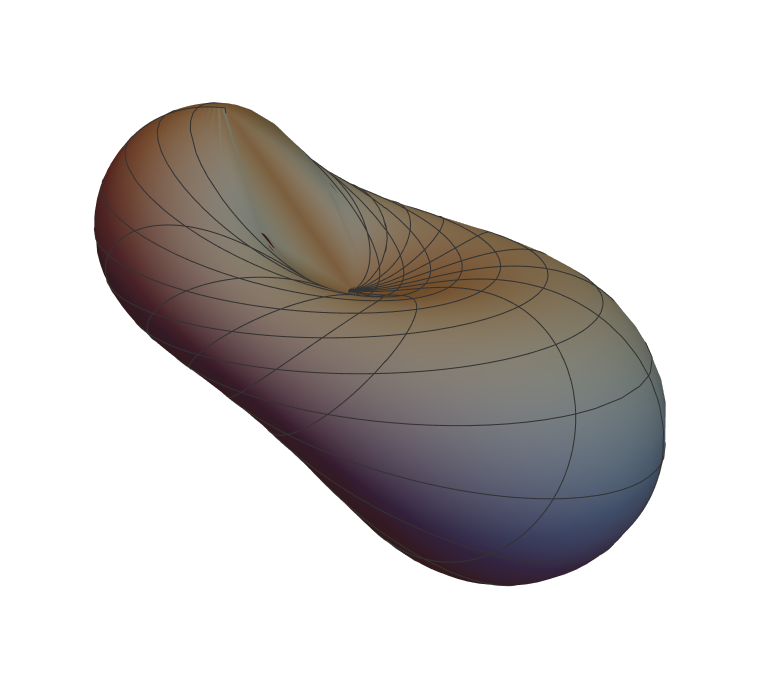}, \includegraphics[width=.4\textwidth]{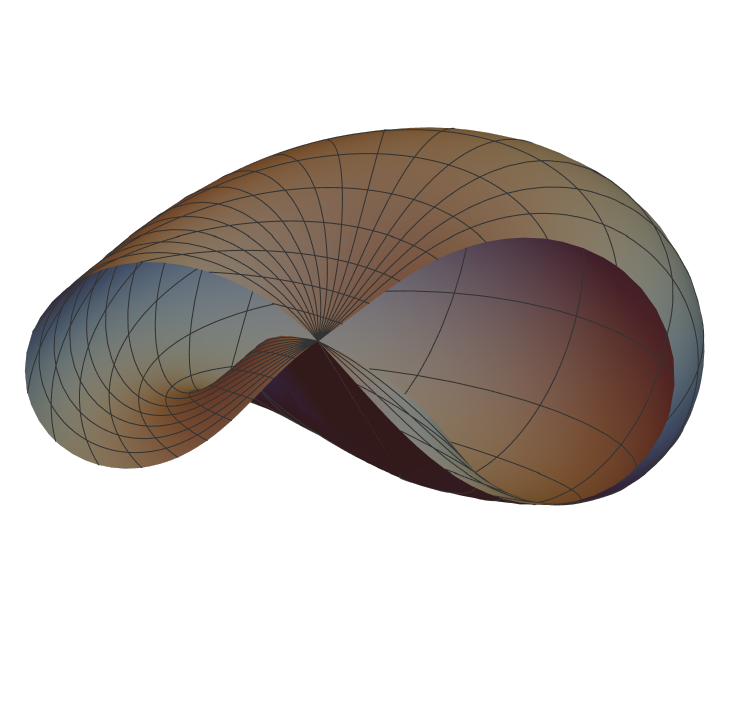}
\end{center}
\caption{Inverse of a Riemann minimal example and a cross-section of that.}\label{fig2}
\end{figure}

Another application of Cor. \ref{c1} is a characterization of the spheres $\s^2_r(r)$ in the family of $-4$-stationary surfaces. We observe  that all these spheres are contained in a vector halfspace of $\r^3$. We prove that this property characterizes these spheres.

\begin{theorem} Spheres passing through the origin $0\in\r^3$ are the only $-4$-stationary surfaces properly immersed in $\r^3$ and contained in a vector halfspace. 
\end{theorem}

\begin{proof} Let $\Sigma$ be a $-4$-stationary surface properly immersed in $\r^3$. After a vector isometry of $\r^3$, we can suppose that $\Sigma$ is contained in the vector halfspace $z>0$. By Cor. \ref{c1}, the surface $\Phi(\Sigma)$ is a minimal surface which it is also proper because inversions preserve proper immersions. Since $\Phi(\Sigma)$ lies contained in the halfspace $z>0$, then $\Phi(\Sigma)$ is a plane by the strong halfspace theorem for minimal surfacest of Hoffman and Meeks \cite{hm}. Since this plane is contained in $z>0$, the equation of this plane is $z=\delta$ for some $\delta>0$. Thus its inversion via $\Phi$, that is, the initial surface $\Sigma$, is a sphere of type $\s_r^2(r)$, proving the result. 
\end{proof}

%%%%%%%%%%%%%%%%%%%%%%%%
\section{The Bj\"{o}rling problem and examples of non-orientable $-4$-stationary surfaces}\label{s6}
%%%%%%%%%%%%%%%%%%%%%%%%%%%%%%%%%

A method of construction of minimal surfaces is the Bj\"{o}rling problem which can be formulated as follows. Let $\alpha\colon I\subset\r\to\r^3$ be a regular   curve and let $V\colon I\to\r^3$ be a unit   vector field along $\alpha$ such that $\langle\alpha'(s),V(s)\rangle=0$ for all $s\in I$. The Bj\"{o}rling problem consists into find a domain $\Omega\subset\r^2$ containing $I$ and a minimal surface $X\colon\Omega \to\r^3$, $X=X(s,t)$, such that $X(s,0)=\alpha(s)$ and $\nu\circ X(s,0)=V(s)$ for all $s\in I$. The answer to this problem is affirmative if the data are analytic, even there is explicit parametrization of the minimal surface \cite{bj,sc}. Since $\Phi$ preserves analytic functions, we prove that the analogue Bj\"{o}rling problem in the family of $-4$-stationary surfaces has a solution.

\begin{theorem} Let $\alpha\colon I\subset\r\to\r^3$ be a regular analytic curve and let $V\colon I\to\r^3$ be a unit analytic vector field along $\alpha$ such that $\langle \alpha'(s),V(s)\rangle=0$ for all $s\in I$.  If $0\not\in\alpha(I)$, then there is a $-4$-stationary surface $\Sigma$ containing $\alpha$ and such that the unit normal of $\Sigma$ along $\alpha$ coincides with $V$.
\end{theorem}

 \begin{proof}
 Define $\widetilde{\alpha}(s)=\Phi(\alpha(s))$ and 
 $$\widetilde{V}(s)=V(s)-2\frac{\langle\alpha(s),V(s)\rangle}{|\alpha(s)|^2}\alpha(s).$$
 It is immediate that $|\widetilde{V}(s)|=1$ and that $\widetilde{V}(s)$ is orthogonal to $\widetilde{\alpha}'(s)$ for all $s\in I$. Let $\widetilde{X}\colon\Omega\to\r^3$, $\widetilde{X}=\widetilde{X}(s,t)$, be the  minimal surface which it is  the Bj\"{o}rling problem for data $\{\widetilde{\alpha},\widetilde{V}\}$. By Cor. \ref{c1}, $X:=\Phi\circ\widetilde{X}\colon\Omega\to\r^3$ is a $-4$-stationary surface. Finally, we check the desired conditions. So, $X(s,0)=\Phi(\widetilde{\alpha}(s))=\Phi^2(\alpha(s))=\alpha(s)$. If we denote by $N$ the unit normal of $\widetilde{X}$, and using \eqref{nn}, we have 
\begin{equation*}
\begin{split}
\nu(X(s,0))&= N(\widetilde{X}(s,0))-2\frac{\langle \widetilde{X}(s,0),N(s,0)\rangle}{|\widetilde{X}(s,0|^2}\widetilde{X}(s,0)\\
 &=\widetilde{V}(s)-2\frac{\langle\beta(s),\widetilde{V}(s)\rangle}{|\beta(s)|^2}\beta(s)=V(s).
 \end{split}
 \end{equation*} 
 \end{proof}

Corollary \ref{c1} allows to give examples of non-orientable $-4$-stationary surfaces because inversions preserve orientability. For this, it is enough to consider the inversion of any non-orientable minimal surface. However, we give an explicit example motivated by a result of Meeks who used the solutions of the Bj\"{o}rling problem to construct minimal surfaces whose topology is that of a M\"{o}bius strip \cite{me}. 
\begin{corollary} There exist $-4$-stationary surfaces whose topology is that of a a M\"{o}bius strip.
\end{corollary}
\begin{proof}
Let $\s^1$ be the circle of radius $1$ and center $0$ situated in the plane $z=0$. If $\alpha(s)=(\cos s,\sin s,0)$ is a parametrization of $\s^1$, let $V(s)=\cos(s/2)\alpha(s)+\sin(s/2)(0,0,1)$. As we go from $s=0$ to $s=2\pi$, the vector field $V$ comes back to the initial position but in reverse position. The corresponding Bj\"{o}rling solution is a M\"{o}bius minimal surface. Therefore, its inverse surface is a $-4$-stationary surface which is not orientable. See Fig. \ref{fig3} 
\end{proof}

 \begin{figure}[hbtp]
\begin{center}
\includegraphics[width=.45\textwidth]{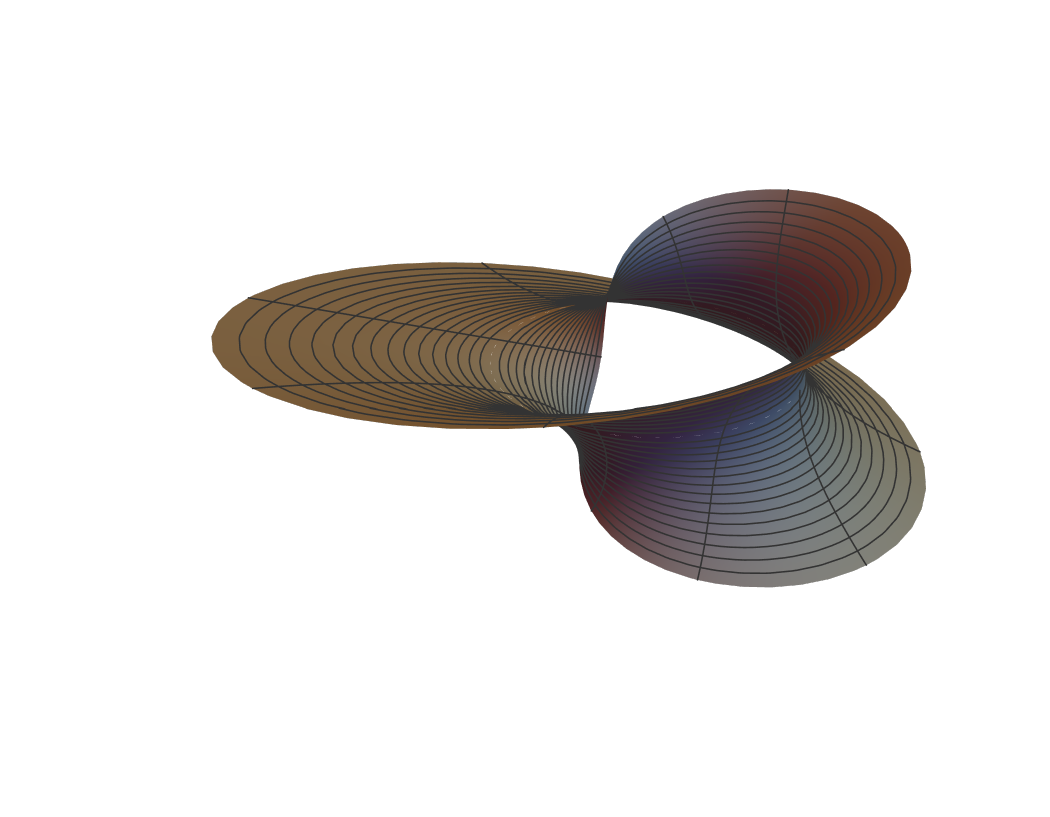}, \includegraphics[width=.3\textwidth]{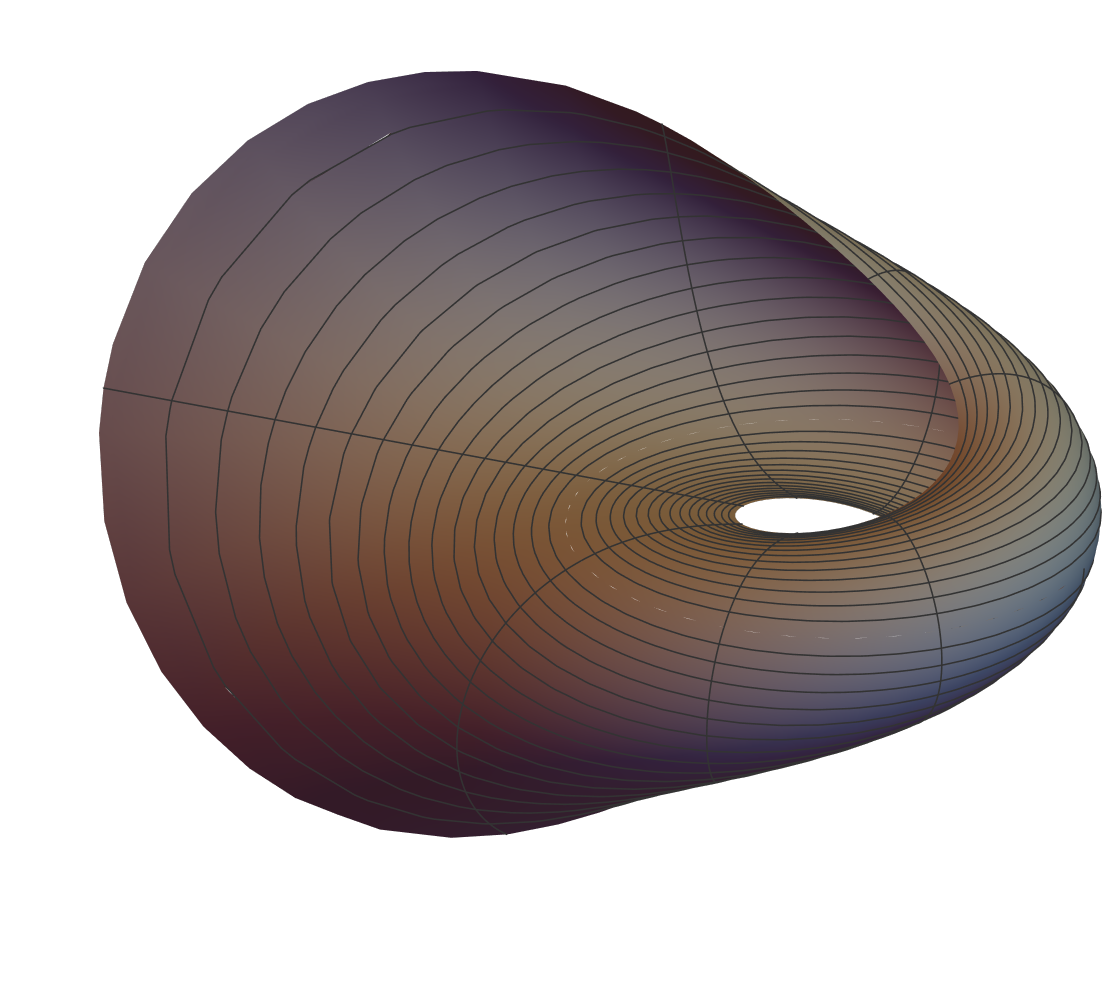}
\end{center}
\caption{M\"{o}bius strips. Left: a minimal surface. Right: a $-4$-stationary surface as inversion of a minimal surface}\label{fig3}
\end{figure}

\section*{Statements and Declarations}
The authors have no conflicts of interest to declare that are relevant to the content of this article.

\section*{Acknowledgements}
The author has been partially supported by MINECO/MICINN/FEDER grant no. PID2023-150727NB-I00, and by the ``Mar\'{\i}a de Maeztu'' Excellence Unit IMAG, reference CEX2020-001105- M, funded by MCINN/AEI/10.13039/ 501100011033/ CEX2020-001105-M.
%%%%%%%%%%%%%%%%%%%%%%%%%%%%%%%%%%%%%%%%%%%%%%%%%%%%%%%%%%%%%%


\begin{thebibliography}{99}
 \bibliographystyle{amsplain}
 
 \bibitem{al} A. D. Alexandrov, A characteristic property of spheres. Ann. Mat. Pura Appl. 58 (1962), 303--315.

\bibitem{ba} V. Bayle, Propri\'et\'es de concavit\'e du profil isop\'erim\'etrique et applications, Graduate thesis, Institut Fourier, Universite Joseph-Fourier-Grenoble I, 2004.

\bibitem{bj}  E. G. Bj\"{o}rling, In integrationem aequationis derivatarum partialium superficiei, cujus in puncto unoquoque principales ambo radii curvedinis aequales sunt signoque contrario. Arch. Math. Phys. 4 (1844), 290--315. 



\bibitem{dh} U. Dierkes, G. Huisken,
The $n$-dimensional analogue of a variational problem of Euler. Math. Ann. 389 (2024), 3841--3863.

\bibitem{dl1} U. Dierkes, R.  L\'opez,  On the Plateau problem for surfaces with minimum moment of inertia. Preprint 2025.
\bibitem{dl2} U. Dierkes, R.  L\'opez, Axisymmetric stationary surfaces for the moment of inertia. Preprint (2025).

\bibitem{en1} A. Enneper, Ueber die cyclischen Fl\"{a}chen. Nach. K\"{o}nigl. Ges. d. Wissensch. G\"{o}ttingen,
Math. Phys. Kl., (1866) 243--249.

\bibitem{en2} A. Enneper, Die cyklischen Fl\"{a}chen. Z. Math. Phys., 14 (1869) 393--421.




\bibitem{eu} L. Euler, Methodus Inveniendi lineas curvas maximi minimive proprietate gaudentes sive solutio problematis isoperimetrici latissimo sensu accepti, Lausanne et Genevae. 1744.

\bibitem{hm} D. Hoffman, W. H. Meeks III, The strong halfspace theorem for minimal surfaces. Invent. Math. 101 (1990), 373--377.


%\bibitem{gt} D. Gilbarg, N. S. Trudinger. Elliptic Partial Differential Equations of Second Order. Classics in Mathematics. Springer-Verlag, Berlin, 2001. Reprint of the 1998 edition.
\bibitem{ka} N. Kapouleas, Compact constant mean curvature surfaces in Euclidean three-space. J. Diff. Geom. 33 (1991), 683--715.
 %\bibitem{ku} W. K\"{u}hnel, M. Steller, On closed Weingarten surfaces, Monatsh. Math. 146 (2005), 113--126.

\bibitem{lo} R. L\'opez, Constant Mean Curvature Surfaces with Boundary, Springer Monographs in
Mathematics, Springer, New York, 2013.

\bibitem{lo1} R. L\'opez,  Ruled surfaces of the least moment of inertia. Preprint (2025).

\bibitem{lo2}   R. L\'opez,   Stationary surfaces for   the moment of  inertia with constant Gauss curvature. Preprint (2025).


 \bibitem{me} W. H. Meeks III, The classification of complete minimal surfaces in $R^3$ with total curvature greater than $-8\pi$. Duke Math. J. 48 (1981), 523--535.

\bibitem{mr} S. Montiel, A. Ros, Curves and Surfaces, second edition, Graduate Studies in Mathematics, vol. 69, American Mathematical Society-Real Sociedad Matem\'atica Espanola, Providence, RI. Madrid, 2009, Translated from the 1998 Spanish original by Montiel and edited by Donald Babbitt.
 
%\bibitem{ps} P. Pucci, J. Serrin. The Maximum Principle. Progress in Nonlinear Differential Equations and their Applications, vol. 73. Birkh\"{a}user Verlag, Basel, 2007. 

\bibitem{ri} B. Riemann, \"{U}ber die Fl\"{a}chen vom kleinsten Inhalt bei gegebener Begrenzung, Abh. K\"{o}nigl.
Ges. d. Wissensch. G\"{o}ttingen, Mathema. Cl. 13 (1868) 329--333.

\bibitem{sc} H. A. Schwarz, Gesammelte Mathematische Abhandlungen, Band I, Springer-Verlag, Berlin 1890, pp. 126--148.

\bibitem{wh} J. H. White, A global invariant of conformal mappings in space. Proc. Amer. Math. Soc. 38 (1973), 162--164.
\end{thebibliography}
\end{document}